\documentclass{cimart}

\title{
    Constructions of well-rounded algebraic lattices over odd prime degree cyclic number fields
    }

\author{
    Robson de Araujo, Antonio de Andrade, Trajano da Nóbrega Neto, Jéfferson Bastos
    }

\authorinfo[
   R. R. de Araujo]{
    Federal Institute of São Paulo, Catanduva, Brazil}{%
    robson.ricardo@ifsp.edu.br
    }

\authorinfo[
    A. A. de Andrade]{
    São Paulo State University, São José do Rio Preto, Brazil}{%
    antonio.andrade@unesp.br
    }

\authorinfo[
    T. P. da Nóbrega Neto]{
    São Paulo State University, São José do Rio Preto, Brazil}{%
    trajano.nobrega@unesp.br
    }

\authorinfo[
    J. L. R. Bastos]{
	São Paulo State University, São José do Rio Preto, Brazil}{%
	jefferson.bastos@unesp.br
}

\abstract{%
    Algebraic lattices are those obtained from modules in the ring of integers of algebraic number fields through canonical or twisted embeddings. In turn, well-rounded lattices are those with maximal cardinality of linearly independent vectors in its set of minimal vectors. Both classes of lattices have been applied for signal transmission in some channels, such as wiretap channels. Recently, some advances have been made in the search for well-rounded lattices that can be realized as algebraic lattices. Moreover, some works have been published that study algebraic lattices obtained from modules in cyclic number fields of odd prime degree $p$. In this work, we generalize some results of a recent work of Tran et al. and we provide new constructions of well-rounded algebraic lattices from a certain family of modules in the ring of integers of each of these fields when $p$ is ramified in its extension over the field of rational numbers.
    }



\VOLUME{33}
\YEAR{2025}
\ISSUE{1}
\NUMBER{6}
\DOI{https://doi.org/10.46298/cm.14235}

\begin{document}

\section{Introduction}


Lattices are discrete additive subgroups of $\mathbb{R}^n$. Recently, they have been considered for applications in different areas, such as coding theory and cryptography \cite{conway,springersueli,lyu-peikert}. Algebraic lattices are those obtained as image in the Euclidean space of some $\mathbb{Z}$-module in the ring of integers of an algebraic number field through the canonical embedding or some twisted embedding. In last decades, algebraic lattices have been studied from different perspectives \cite{eva,andrade-0,grasi-1,ibilce-1}.

Explicitly, a lattice $\Lambda \subseteq \mathbb{R}^n$ of rank $k\leq n$ is defined as the $\mathbb{Z}$-module generated by a set $\mathcal{B}=\{u_1,u_2,\ldots,u_k\}$ of $k$ linearly independent vectors in $\mathbb{R}^n$ - this set $\mathcal{B}$ is called a basis of $\Lambda$. In this work, we only consider full-rank lattices, which are those having maximal rank $k=n$. If $\Lambda$ is a full-rank lattice in $\mathbb{R}^n$, then it can be obtained as $\Lambda=\textbf{M}\mathbb{Z}^n$, where $\textbf{M}$ is the matrix $n\times n$ whose columns are given by the entries of the vector in a basis of $\Lambda$ - this matrix is called a generator matrix of $\Lambda$. In this case, the volume of $\Lambda$ is defined by $Vol(\Lambda) = |\det\left(\textbf{M}\right)|$ and the minimum norm of $\Lambda$ is given by $t_{\Lambda} = \min\left\{ \|\textbf{u} \|^2:\textbf{0}\neq \textbf{u}\in\Lambda \right\}$, where $\|.\|$ is the usual Euclidian norm in $\mathbb{R}^n$. The center density of $\Lambda$ is defined as $\delta(\Lambda) = \rho(\Lambda)^n/Vol(\Lambda)$, where $\rho(\Lambda)=\sqrt{t_{\Lambda}}/2$ is the largest radius such that it is possible to obtain a sphere packing with centers in the points of the lattice $\Lambda$. This parameter $\delta(\Lambda)$ is important because it is related to the classic sphere packing problem \cite{conway}, since the center density is greater, the spherical packing centered at the points of the lattice is greater.

The set of minimal vectors of a lattice $\Lambda$ is defined by $S(\Lambda)=\{\textbf{u}\in\Lambda:~\|\textbf{u}^2 \|=t_\Lambda\}$. The lattice $\Lambda$ is said to be well-rounded if $S(\Lambda)$ generates $\mathbb{R}^n$, that is, if $S(\Lambda)$ contains a subset of $n$ linearly independent vectors (this set can be or not a basis of the lattice). Research on well-rounded lattices has recently been developed due to their important properties and their applications for signal transmission over SISO and MIMO channels \cite{Martinet2003,wrsiso,wrmimo}. 

In particular, studies linking algebraic lattices and well-rounded lattices have been made after the remarkable work of Fukshansky and Petersen in 2012 \cite{petersen}. In this work, the authors provide well-rounded algebraic lattices via real quadratic fields and prove a necessary and sufficient condition for the algebraic lattice coming from the whole ring of integers of an algebraic number field via the Minkowski embedding be well-rounded. Araujo and Costa \cite{robson-2} obtained (infinitely many) well-rounded lattices from cyclic number fields of odd prime degree in the unramified case. In the last years, several other articles have been published relating well-rounded and algebraic lattices, such as \cite{fukii,damir,srini,damir2022,tran}.

Recently, several papers have been published studying algebraic lattices coming from $\mathbb{Z}$-modules in the ring of integers of cyclic number fields of odd prime degree $p$ via the canonical embedding \cite{everton,josewalter,robson-1,robson-2,carmelo-ijam,ibilce-1}, some of them in the perspective of the well-roundedness property. In this context, we need to consider two different cases: when $p$ is unramified and when $p$ in ramified in the extension of the fixed number field over the field of rational numbers. A family with infinitely many well-rounded algebraic lattices was presented in \cite{robson-2}. In turn, some constructions of well-rounded algebraic lattices have been provided in the ramified case in \cite{ibilce-2}.

In this work, we present new constructions of well-rounded algebraic lattices in the ramified case (Section \ref{subsec_main}). In Proposition \ref{toni-1}, we generalize the result of \cite[Lemma 2.5]{tran}, which is related to the construction of well-rounded algebraic lattices in cubic number fields. Using this fact, we provided a way to obtain well-rounded algebraic lattices over $\mathbb{K}$ via the canonical embedding (Corollary \ref{cor_X}). Moreover, we present a family of modules over $\mathbb{K}$ which realizes well-rounded lattices via the canonical embedding (Subsection \ref{subsec_main}) and we give some additional results.


This paper is organized as follows. In Section \ref{sec_prel}, we present some definitions and basic facts about algebraic lattices (Subsection \ref{sec-2}) and about odd prime degree cyclic number fields (Subsection \ref{sec-3}). In Section \ref{sec_MM}, we present the contributions mentioned in the last paragraph.

\section{Preliminaries}\label{sec_prel}

In this section we present some definitions and facts related to algebraic lattices (Subsection \ref{sec-2}) and to odd prime degree cyclic number fields (Subsection \ref{sec-3}) necessary in the development of this article.

\subsection{Algebraic lattices} \label{sec-2}

Let $\mathbb{K}$ be an algebraic number field of degree $n$ and be $\mathcal{O}_{\mathbb{K}}$ its ring of algebraic integers. There are exactly $n$ distinct $\mathbb{Q}$-monomorphisms $\sigma_i:\mathbb{K}\rightarrow\mathbb{C}$, for $i=1,2,\ldots,n$. A $\mathbb{Q}$-monomorphism $\sigma_i$ is said to be real if $\sigma_i(\mathbb{K})\subseteq\mathbb{R}$, and imaginary otherwise. A number field $\mathbb{K}$ is said to be totally real if $\sigma_i$ is real for all $i=1,2,\ldots,n$ and totally imaginary if $\sigma_i$ is imaginary for all $i=1,2,\ldots,n$. If $r_1\geq 0$ denotes the number of indices such that $\sigma_i(\mathbb{K})\subset \mathbb{R}$, then $n-r_1$ is an even number satisfying $r_1+2r_2=n$. In order to standardize, we denote the $\mathbb{Q}$-monomorphisms $\sigma_1,\sigma_2,\ldots,\sigma_n$ in such a way that $\sigma_1,\ldots,\sigma_{r_1}$ are the real $\mathbb{Q}$-monomorphisms and that $\sigma_{r_1+r_2+j}=\overline{\sigma_{r_1+j}}$, for $j=1,2,\ldots,r_2$. 

The trace of any element $\alpha\in\mathbb{K}$ is defined to be the rational number 
\[
Tr_{\mathbb{K}}(\alpha)=\sum_{i=1}^n\sigma_i(\alpha)
\]
and the discriminant of $\mathbb{K}$ over $\mathbb{Q}$ is given by 
\[
D(\mathbb{K})= \det(Tr_{\mathbb{K}}(\alpha_i\alpha_j))_{i,j=1}^n,
\]
where $\{\alpha_1,\alpha_2,\ldots,\alpha_n\}$ is an integral basis of $\mathcal{O}_{\mathbb{K}}$. The canonical embedding $\sigma:\mathbb{K}\rightarrow\mathbb{R}^n$ is defined by setting $\sigma(x)$ as
\begin{equation}
	(\sigma_1(x),\ldots,\sigma_{r_1}(x),Re(\sigma_{r_1+1}(x)),Im(\sigma_{r_1+1}(x)),\ldots,Re(\sigma_{r_1+r_2}(x)),Im(\sigma_{r_1+r_2}(x))),
\end{equation}
where $x\in\mathbb{K}$, and $Re(\beta)$ and $Im(\beta)$ denote the real and the imaginary parts of the complex number $\beta$, respectively \cite{samuel}.

If $\mathcal{M}$ is a free $\mathbb{Z}$-module of $\mathcal{O}_{\mathbb{K}}$ with rank $n$, then $\Lambda=\sigma(\mathcal{M})$ is an $n$-dimensional lattice whose minimum is given by $t_{\Lambda}=\min\{\|\sigma(x)\|^2:x\in\mathcal{M},\ x\neq 0\}$, where
$$ \|\sigma(x)\|^2 = \left\{ \begin{array}{ll} Tr_{\mathbb{K}}(x^2) & \mbox{if} \ \mathbb{K} \ \mbox{is totally real};\\  \frac{1}{2}Tr_{\mathbb{K}}(x\overline{x}) & \mbox{if} \ \mathbb{K} \ \mbox{is totally complex}. \end{array} \right.$$
if $\mathbb{K}$ is an Abelian number field. The lattice $\Lambda$ is called an algebraic lattice. In particular, if $\mathcal{M}$ is an integral ideal of $\mathcal{O}_\mathbb{K}$, $\Lambda$ is called an ideal lattice. The center density of the algebraic lattice $\Lambda=\sigma(\mathcal{M})$ is given by
\begin{equation} \label{densidade}
	\delta(\Lambda) = \frac{(\sqrt{t_{\Lambda}}/2)^n}{[\mathcal{O}_{\mathbb{K}}:\mathcal{M}]\sqrt{|D(\mathbb{K})|}} = \frac{t_{\Lambda}^{n/2}}{2^n[\mathcal{O}_{\mathbb{K}}:\mathcal{M}]\sqrt{|D(\mathbb{K})|}}, \end{equation}    
where $[\mathcal{O}_{\mathbb{K}}:\mathcal{M}]$ denotes the index of $\mathcal{M}$ in $\mathcal{O}_{\mathbb{K}}$ as additive groups \cite{samuel}.

\subsection{Odd prime degree cyclic number fields}  \label{sec-3}

Let $\mathbb{K}$ be a cyclic number field of prime degree $p>2$. This means that $\mathbb{K}/\mathbb{Q}$ is an Abelian extension of degree $p$. Also, $\mathbb{K}$ is a totally real number field. By Kronecker-Weber Theorem, there exists $n>0$ such that $\mathbb{K}\subseteq \mathbb{Q}(\zeta_n)$, where $\zeta_n$ is a primitive $n$-th root of unity \cite[Theorem 14.1]{was}. The smallest $n$ with this property is called the conductor of $\mathbb{K}$. The discriminant of $\mathbb{K}$ is given by $D(\mathbb{K})=n^{p-1}$ \cite{trajano-1}. It is well known (see \cite{spearman}, for example) that: \begin{enumerate}
	\item $p$ is unramified in $\mathbb{K}$ if and only if $n=p_1p_2\ldots p_s$, with $s\geq 1$, or \item $p$ is ramified in $\mathbb{K}$ if and only if $n=p^2p_1p_2\ldots p_s$, with $s\geq 0$, \end{enumerate} where $p_i$ are distinct prime numbers satisfying $p_i\equiv 1~(mod\ p)$, for $i=1,2,\ldots,s$. Furthermore:
\begin{enumerate} \item if $p$ is unramified in $\mathbb{K}$, then $p\mathcal{O}_{\mathbb{K}}=\mathcal{B}$ and $p_i\mathcal{O}_{\mathbb{K}}=\mathcal{B}_i^p$, or \item if $p$ is ramified in $\mathbb{K}$, then $p\mathcal{O}_{\mathbb{K}} = \mathcal{B}^p$ and $p_i\mathcal{O}_{\mathbb{K}} = \mathcal{B}_i^p$, \end{enumerate} where $\mathcal{B}$ and $\mathcal{B}_i$ are prime ideals in $\mathcal{O}_{\mathbb{K}}$ such that $\mathcal{B}\cap \mathbb{Z}=p\mathbb{Z}$ and $\mathcal{B}_i\cap \mathbb{Z}=p_i\mathbb{Z}$, for $i=1,2,\ldots,s$.

Denote by $\theta$ a generator of the cyclic Galois group $Gal(\mathbb{K}/\mathbb{Q})$ and by $t=Tr_{\mathbb{Q}(\zeta_n)/\mathbb{K}}(\zeta_n)$ the trace of $\zeta_n$ in the field extension $\mathbb{Q}(\zeta_n)/\mathbb{K}$. As shown in \cite{everton,trajano-2,robson-1}, it is known that:
\begin{enumerate} \item if $p$ is unramified in $\mathbb{K}$, then $\{t,\theta(t),\ldots,\theta^{p-1}(t)\}$ is an integral basis of $\mathbb{K}$ and $Tr_{\mathbb{K}}(\theta^i(t)) = 
	(-1)^{s}$, for $i=0,1,\ldots,p-1$, and
	\item if $p$ is ramified in $\mathbb{K}$, then $\{1,\theta(t),\ldots,\theta^{p-1}(t)\}$ is an integral basis of $\mathbb{K}$ and also $Tr_{\mathbb{K}}(\theta^i(t))=0$, for $i=0,1,\ldots,p-1$. \end{enumerate}

\section{Well-rounded algebraic lattices}\label{sec_MM}

Let $\mathbb{K}$ be a cyclic number field of prime degree $p>2$. Consider the notation adopted in Subsection \ref{sec-3}. In this section we present some constructions of well-rounded algebraic lattices coming from $\mathbb{Z}$-modules in the ring of integers of $\mathbb{K}$. Firstly, we generalize to $p$-th degree a result presented in \cite{tran} for third degree. We start with a technical lemma:

\begin{lemma} \label{le-1} Let $\zeta_p$ be a primitive $p$-th root of unity and $\alpha\in \mathcal{O}_\mathbb{K}$. Consider the polynomial $f(x)=\alpha+\theta(\alpha)x+\theta^2(\alpha)x^2+\ldots+\theta^{p-1}(\alpha)x^{p-1} \in \mathcal{O}_\mathbb{K}[x]$. If $\alpha\in\mathcal{O}_{\mathbb{K}}\setminus\mathbb{Z}$, then $f(\zeta_p^i)\neq 0$, for all $i=1,2,\ldots,p-1$. \end{lemma}
\begin{proof} Let $\phi(x)=1+x+x^2+\ldots+x^{p-1}\in\mathbb{Z}[x]$ be the minimal polynomial of $\zeta_p$. Since $\gcd(p,p-1)=1$, then $\mathbb{L}=\mathbb{K}[\zeta_p]$ has degree $p(p-1)$, which implies that $\phi(x)$ is irreducible over $\mathbb{K}$. So, supposing that $f(\zeta_p^i)=0$ for some $i=1,2,\ldots,p-1$, then $\phi(x)$ divides $f(x)$ in $\mathbb{K}[x]$ since $\zeta_p^i$ is a root of the polynomials $f(x)$ and $\phi(x)$ simultaneously. Thus, since $\phi(x)$ is irreducible in $\mathbb{K}[x]$ and have the same degree of $f(x)$, it follows that $f(x)=\alpha\phi(x)$. This implies that $\theta^j(\alpha)=\alpha$ for all $j=1,\ldots,p-1$. However, this leads to $\alpha \in \mathbb{Z}$, which is a contradiction. Therefore, $f(\zeta_p^i)\neq 0$, for all $i=1,2,\ldots,p-1$.
\end{proof}


The next result extends Lemma 2.5 of \cite{tran}:

\begin{proposition} \label{toni-1} Let $\alpha\in\mathcal{O}_{\mathbb{K}}\setminus\mathbb{Z}$. Then, $Tr_{\mathbb{K}}(\alpha)\neq 0$ if and only if the set $$\{\sigma(\alpha),\sigma(\theta(\alpha)),\ldots,\sigma(\theta^{p-1}(\alpha))\}$$ is a $\mathbb{R}$-linearly independent subset of $\mathbb{R}^p$, where $\sigma$ is the canonical embedding of $\mathbb{K}$ in $\mathbb{R}^{p}$. \end{proposition}

\begin{proof} Suppose
	$$a_0\sigma(\alpha)+a_1\sigma(\theta(\alpha))+\cdots+a_{p-1}\sigma(\theta^{p-1}(\alpha))=0,$$
	where $a_0,a_1,\ldots,a_{p-1} \in\mathbb{R}$. Since $\sigma(x)=(x,\theta(x),\ldots,\theta^{p-1}(x))$, where $x\in\mathbb{K}$, it follows that 
	$$C\cdot \left(\begin{array}{cccc} a_0& a_1& \ldots & a_{p-1} \end{array} \right)^T = \left( \begin{array}{cccc} 0 & 0 & \cdots& 0 \end{array} \right)^T$$
	where $C$ is the circulant matrix 
	$$C = \left( \begin{array}{ccccc} \alpha & \theta(\alpha) & \theta^2(\alpha) & \cdots & \theta^{p-1}(\alpha)\\ \theta^{p-1}(\alpha) & \alpha & \theta(\alpha) & \cdots & \theta^{p-2}(\alpha)\\ \vdots & \vdots & \vdots & \ddots & \vdots\\ \theta(\alpha) & \theta^{2}(\alpha) & \theta^{3}(\alpha) & \cdots & \alpha \end{array}\right).$$
	It is well-known that the determinant of $C$ is given by 
	$\det(C)= \displaystyle\prod_{i=0}^{p-1}f(\zeta_p^i)$, where 
    \[
    f(\zeta_p^i)=\alpha+\theta(\alpha)\zeta_p^i+\theta^{2}(\alpha)\zeta_p^{2i}+\cdots+\theta^{p-1}(\alpha)\zeta_p^{(p-1)i},
    \]
    for $i=0,1,\ldots,p-1$, and $\zeta_p$ is a primitive $p$-th root of unity. So, $\det(C)\neq 0$ if and only if $f(\zeta_p^i)\neq 0$, for all $i=0,1,\ldots,p-1$. From Lemma \ref{le-1}, it follows that $f(\zeta_p^i)\neq 0$, for $i=1,2,\ldots,p-1$. Thus, $\det(C)\neq 0$ if and only if $f(1)=Tr_{\mathbb{K}}(\alpha)\neq 0$. Therefore, $B=\{\sigma(\alpha),\sigma(\theta(\alpha)),\ldots,\sigma(\theta^{p-1}(\alpha))\}$ is $\mathbb{R}$-linearly independent subset of $\mathbb{R}^p$ if and only if $Tr_{\mathbb{K}}(\alpha)\neq 0$. \end{proof}

The following corollary presents the construction of some well-rounded algebraic lattices via some special submodules of $\mathcal{O}_\mathbb{K}$:

\begin{corollary}\label{cor_X} Let $\mathcal{M}\subseteq \mathcal{O}_{\mathbb{K}}$ be a $\mathbb{Z}$-module such that $\theta(\mathcal{M})\subseteq \mathcal{M}$. Let $\alpha\in \mathcal{M}\setminus\mathbb{Z}$ such that $\sigma(\alpha)$ is one of the shortest vectors in the lattice $\Lambda=\sigma(\mathcal{M})$. Then $Tr_{\mathbb{K}}(\alpha)\neq 0$ if and only if $B=\{\sigma(\alpha),\sigma(\theta(\alpha)),\ldots,\sigma(\theta^{p-1}(\alpha))\}$ generates a well-rounded sublattice of $\Lambda$ of rank $p$. \end{corollary}

\begin{proof} Firstly, we note that $\|\sigma(\theta^i(\alpha))\|=\lambda_1$ is a constant number for all $i=0,1,\ldots,p-1$, because, in this case, the canonical embedding is given by 
\[
\sigma(x)=(x,\theta(x),\theta^2(x),\ldots,\theta^{p-1}(x)),
\]
which leads to the fact that the coordinates of $\sigma(\theta^i(\alpha))$ are a permutation of that of $\sigma(\alpha)$. Since $\theta(\mathcal{M})\subseteq\mathcal{M}$ by hypothesis, and so $\theta^i(\mathcal{M})\subseteq \mathcal{M}$ for all $i=0,1,\ldots,p-1$, then $L=\langle B \rangle_\mathbb{Z}$ generates a sublattice of $\Lambda$ containing a shortest vector of it. Thus, $B$ is a set having only minimal vectors of $L$. From this fact and from Proposition \ref{toni-1}, finally this leads to the fact that $L=\langle B\rangle_\mathbb{Z}$ is a well-rounded full-rank sublattice of $\Lambda$ if and only if $Tr_\mathbb{K}(\alpha)\neq 0$.
\end{proof}

We remark that the hypothesis in Corollary \ref{cor_X} given by $\theta(\mathcal{M})\subseteq \mathcal{M}$ happens, for example, if $\mathcal{M}=\mathcal{O}_\mathbb{K}$ or if $\mathcal{M}=\mathcal{B}_1\mathcal{B}_2\ldots\mathcal{B}_s$, since $\theta(\mathcal{B}_i)=\mathcal{B}_j$, for all $i,j=1,2,\ldots,s$.

\subsection{The unramified case}\label{sec-unra}

Now suppose that $p$ is unramified in $\mathbb{K}/\mathbb{Q}$. So the conductor of $\mathbb{K}$ is $n=p_1p_2\ldots p_s$, for $s\geq 1$. For any $m>0$ integer number, consider the subset of $\mathcal{O}_\mathbb{K}$ given by
\[
\mathcal{M}_m=\{\alpha\in\mathcal{O}_{\mathbb{K}}: Tr_{\mathbb{K}}(\alpha) \equiv 0\pmod{m}\}, 
\]
or, equivalently, 
$$\mathcal{M}_m=\left\{\alpha=\sum_{i=0}^{p-1} a_i\theta^i(t)\in\mathcal{O}_{\mathbb{K}}: a_0,\ldots,a_{p-1}\in\mathbb{Z},~\sum_{i=1}^{p-1}a_i \equiv 0\pmod{m}\right\}.$$
As observed in \cite{ibilce-1}, these modules generalize the prime ideals of $\mathcal{O}_\mathbb{K}$ above $p_i$ - in fact, $\mathcal{B}_j=\mathcal{M}_{p_j}$, for $j=1,2,\ldots,s$. Additionally, in the following proposition we give an explicit characterization of each prime ideal $\mathcal{B}_i$:
\begin{proposition} \label{ana-2} $\displaystyle \mathcal{B}_j=p_j\mathbb{Z}t+\sum_{i=1}^{p-1}\mathbb{Z}(\theta^i(t)-t)$. \end{proposition} \begin{proof} Let $\alpha=\sum_{i=0}^{p-1}a_i\theta^i(t)\in \mathcal{O}_{\mathbb{K}}$. Thus, $\alpha\in\mathcal{B}_j$ if and only if $\sum_{i=0}^{p-1}a_i=p_jk$, for some $k\in\mathbb{Z}$. Thus,
	\begin{equation*}
		\alpha=\sum_{i=0}^{p-1}a_i\theta^i(t)-\sum_{i=0}^{p-1}a_it+\sum_{i=0}^{p-1}a_it = t\sum_{i=0}^{p-1}a_i+\sum_{i=0}^{p-1}a_i(\theta^i(t)-t) = p_jkt+\sum_{i=1}^{p-1}a_i(\theta^i(t)-t).
	\end{equation*} Therefore, $\alpha \in \mathcal{B}_j$ if and only if $\alpha \in \displaystyle p_j\mathbb{Z}t+\sum_{i=1}^{p-1}\mathbb{Z}(\theta^i(t)-t)$. \end{proof}

Furthermore, about the family of $\mathbb{Z}$-modules $\mathcal{M}_m$, from \cite{robson-2,ibilce-1}, it follows that:
\begin{enumerate} 
	\item $\mathcal{M}_m$ is a $\mathbb{Z}$-module of index $m$ and rank $p$ in $\mathcal{O}_{\mathbb{K}}$; 
	\item If $m\equiv 1 \pmod{p}$, the lattice $\sigma(\mathcal{M}_m)$ is well-rounded if and only if 
    \[
    \sqrt{\frac{n}{p+1}}\leq m \leq \sqrt{n(p+1)};
    \]
	\item An element $\displaystyle \alpha\in \mathcal{O}_{\mathbb{K}}$ belongs to $\mathcal{B}_1\mathcal{B}_2\ldots\mathcal{B}_s$ if and only if $\displaystyle Tr_{\mathbb{K}}(\alpha)\equiv 0 \pmod{n}$; \item $\mathcal{M}_m$ is an ideal of $\mathcal{O}_{\mathbb{K}}$ if and only if $m|n$. \end{enumerate}

We emphasize the second point mentioned above, which states that $\sigma(\mathcal{M}_m)$ is well-rounded under certain conditions on $m$ and $p$. Next, we explore the well-roundedness property of a similar family of $\mathbb{Z}$-modules in the case where $p$ is ramified in $\mathbb{K}/\mathbb{Q}$.

\subsection{The ramified case}\label{sec_main}

In this section, our objective is to construct well-rounded algebraic lattices in the ramified case via $\mathbb{Z}$-submodules of $\mathcal{O}_\mathbb{K}$ with similar characterization of that presented in the unramified case (Subsection \ref{sec-unra}). Suppose that $p$ is ramified in the extension $\mathbb{K}/\mathbb{Q}$, that is, $p\mathcal{O}_{\mathbb{K}}=\mathcal{B}^{p}$. In this case, the conductor of $\mathbb{K}$ is $n=p^2p_1p_2\ldots p_s$, for $s\geq 0$. We first give a characterization of the prime ideal above $p$ using the norm function $N_{\mathbb{L}/\mathbb{K}}$ in the extension $\mathbb{L}=\mathbb{Q}(\zeta_{p^2})$ over $\mathbb{K}$, i.e., $n=p^2$:

\begin{proposition} If the conductor of $\mathbb{K}$ is $n=p^2$, then 
\[
p\mathcal{O}_{\mathbb{K}}=\mathcal{B}^{p},
\]
where $\mathcal{B}=\langle N_{\mathbb{L}/\mathbb{K}}(1-\zeta_n)\rangle$. 
\end{proposition}

\begin{proof} In this case, it is well-known that $p\mathcal{O}_{\mathbb{L}}= \mathcal{B}_{\mathbb{L}}^{p(p-1)}$, where $\mathcal{B}_{\mathbb{L}}=(1-\zeta_n)\mathcal{O}_{\mathbb{L}}$ \cite[(10.1) Lemma]{neukirch}. Let $\lambda = N_{\mathbb{L}/\mathbb{K}}(1-\zeta_n) = \prod_{j=1}^{p-1}(1-\zeta_n^{r^{jp}})$, where $r$ is a generator of the cyclic group $\left(\mathbb{Z}/p^2\mathbb{Z}\right)^*$ of the inversible elements of $\mathbb{Z}/p^2\mathbb{Z}$. Since $1-\zeta_n$ is a conjugate of $1-\zeta_n^{r^{jp}}$, for some $j=1,2,\ldots,p-1$, it follows that $(1-\zeta_n)\mathcal{O}_{\mathbb{L}}=(1-\zeta_n^{r^{jp}})\mathcal{O}_{\mathbb{L}}$. So,
	$$\prod_{j=1}^{p-1}(1-\zeta_n^{r^{jp}})\mathcal{O}_{\mathbb{L}} = (1-\zeta_n)^{p-1}\mathcal{O}_{\mathbb{L}},$$
	and so $\lambda\mathcal{O}_{\mathbb{L}}=(1-\zeta_n^{r^{jp}})\mathcal{O}_{\mathbb{L}}$. Thus, $\lambda\mathcal{O}_{\mathbb{L}} = \mathcal{B}_{\mathbb{L}}^{p-1} = \mathcal{B}\mathcal{O}_{\mathbb{L}}$. Therefore,
	$\lambda\mathcal{O}_{\mathbb{L}} = \mathcal{B}\mathcal{O}_{\mathbb{L}}$, that is, $\mathcal{B} = \langle \lambda\rangle = \langle N_{\mathbb{L}/\mathbb{K}}(1-\zeta_n)\rangle$, which proves the result. \end{proof}

In the following, we present a family of $\mathbb{Z}$-submodules of $\mathcal{O}_\mathbb{K}$ initially studied in \cite{ibilce-1}. For any positive integers $m,c\in\mathbb{Z}$ such that $0\leq c <m$, consider the set
$$\mathcal{M}_{m,c} = \left\{\sum_{i=0}^{p-1} a_i\theta^i(t)\in\mathcal{O}_{\mathbb{K}}~:~ a_0+c\sum_{i=1}^{p-1}a_i\equiv 0 \pmod{m} \right\}$$
(where the coefficients $a_i$ are integer numbers). From \cite{ibilce-1}, it follows that 
\begin{enumerate} 
\item $\displaystyle \mathcal{M}_{m,c} = \left\{\alpha\in\mathcal{O}_\mathbb{K}~:~Tr_{\mathbb{K}}\left(\frac{1}{p}\alpha-\frac{pc}{n}\alpha t\right) \equiv 0(mod\ m) \right\}$;
	\item $S=\{m,c-\theta(t),c-\theta^2(t),\ldots,c-\theta^{p-1}(t)\}$ is a $\mathbb{Z}$-basis of $\mathcal{M}_{m,c}$; \item $\mathcal{M}_{m,c}$ has rank $p$; \item $\mathcal{O}_\mathbb{K}/\mathcal{M}_{m,c}\cong \mathbb{Z}/p\mathbb{Z}$;
	\item $[\mathcal{O}_\mathbb{K}:\mathcal{M}_{m,c}]=m$; \item If $i\in\{1,\ldots,p-1\}$, then $\mathcal{B}_i=\mathcal{M}_{p_i,0}$ and $\mathcal{B}=\mathcal{M}_{p,\ell}$ for some $\ell\in\{0,\ldots,p-1\}$ such that $t-l\in \mathcal{B}$; \item If $\alpha=a_0m+a_1(c-\theta(t))+\cdots+a_{p-1}(c-\theta^{p-1}(t))\in \mathcal{M}_{m,c}$, with $a_i\in\mathbb{Z}$, for $i\in\{0,\ldots,p-1\}$, then 
	\begin{equation}
		Tr_{\mathbb{K}}(\alpha^2) = p\left(\left(a_0m+c\sum_{i=1}^{p-1}a_i\right)^2 + u \left(p\sum_{i=1}^{p-1} a_i^2 - \left(\sum_{i=1}^{p-1}a_i\right)^2 \right)\right),
	\end{equation} where $u=n/p^2$. 
    \end{enumerate} 

In this work, we specifically focus on the submodule $\mathcal{M}_{m}\subseteq \mathcal{O}_{\mathbb{K}}$ for any rational integer $m>1$:
$$\mathcal{M}_m:= \left\{\alpha\in\mathcal{O}_{\mathbb{K}}: Tr_{\mathbb{K}}(\alpha)\equiv 0\pmod{m}\right\}.$$ We observe that \begin{enumerate} \item $\mathcal{M}_m=\mathcal{M}_{m,0}$ if and only if $p\nmid m$; \item $\mathcal{M}_m=\mathcal{M}_{m/p,0}$ if and only if $p\mid m$. \end{enumerate} 

\begin{lemma} \label{ideal-2} $\mathcal{M}_m$ is an ideal of $\mathcal{O}_{\mathbb{K}}$ if and only if $m\mid n$. \end{lemma}
\begin{proof} If $\alpha=a_0+a_1\theta(t)+\cdots+a_{p-1}\theta^{p-1}(t)\in\mathcal{M}_m$ (with $a_0,a_1,\ldots,a_{p-1}\in\mathbb{Z}$), then $Tr_{\mathbb{K}}(\alpha) \equiv 0 \pmod{m}$. From \cite[Theorem 3.1]{robson-1}, it follows that \begin{equation} \label{eq-0}  Tr_{\mathbb{K}}(\theta^i(t)\theta^j(t)) = Tr_{\mathbb{K}}(t\theta^{i-j}(t)) = \left\{ \begin{array}{l} \frac{n(p-1)}{p} \ \ \mbox{if} \ \ i=j\\ -\frac{n}{p} \ \ \mbox{if} \ \ i\neq j, \end{array} \right. \end{equation}  for $i,j=0,1,\ldots,p-1$. Thus, \[ Tr_{\mathbb{K}}(\alpha\theta^k(t))=(a_1+\cdots+a_{p-1})\left(\frac{-n}{p}\right) + a_kn, \] for $k=1,\ldots,p-1$.
If $m\mid n$, then $Tr_{\mathbb{K}}(\alpha\theta^k(t))\equiv 0\pmod{m}$, for $k=1,\ldots,p-1$. Thus, $\alpha\theta^k(t)\in \mathcal{M}_m$, for $k=1,\ldots,p-1$. Since the set $\{1,\theta(t),\ldots,\theta^{p-1}(t)\}$ is an integral basis of $\mathbb{K}$, it follows that $\mathcal{M}_m$ is an ideal. Reciprocally, $Tr_{\mathbb{K}}(\theta(t)-\theta^2(t))=0$, and, therefore, $\theta(t)-\theta^2(t)\in \mathcal{M}_m$. Since $\theta(t)\in \mathcal{O}_{\mathbb{K}}$ and $\mathcal{M}_m$ is an ideal, it follows that $\theta(t)(\theta(t)-\theta^2(t))\in\mathcal{M}_m$. From Equation \eqref{eq-0}, it follows that 
\[
Tr_{\mathbb{K}}(\theta(t)(\theta(t)-\theta^2(t)))=Tr_{\mathbb{K}}(t^2)-Tr_{\mathbb{K}}(t\theta(t)) = n\equiv 0 \pmod{m},
\]
that is, $m\mid n$. 
\end{proof}

\begin{lemma} The index $[\mathcal{O}_{\mathbb{K}}:\mathcal{M}_m]$ is $m$. \end{lemma} \begin{proof} Let $\alpha\in \mathcal{O}_{\mathbb{K}}$ and $[\alpha]$ denote the coset of $M_m$ in $\mathcal{O}_{\mathbb{K}}$ containing $\alpha$. The proof is completed by showing that the cosets $[0],[1],[2],\ldots,[(m-1)]$ partition $\mathcal{O}_{\mathbb{K}}$. Indeed,
let $0\leq i\leq j\leq m-1$. Then $[i]=[j]$ if and only if $[(i-j)] = [0]$, that is, $i-j \equiv 0 \pmod{m}$, whence $i=j$ and the cosets $[0],[1],[2],\ldots,[(m-1)]$ are distinct. Finally, let $\displaystyle \alpha= a_0+\sum_{i=1}^{p-1}a_i\theta^i(t)\in \mathcal{O}_{\mathbb{K}}$ (with $a_0,a_1,\ldots,a_{p-1}$ integer numbers). We can write
\[
\alpha = a_0 + \sum_{i=1}^{p-1}a_i\theta^i(t) = a_0 + m\sum_{i=1}^{p-1} a_i + \sum_{i=1}^{p-1} a_i(\theta^i(t)-m).
\]
Since $\displaystyle m\sum_{i=1}^{p-1}a_i+\sum_{i=1}^{p-1}a_i(\theta^i(t)-m) \in \mathcal{M}_m$, it follows that $\alpha \equiv a_0\pmod{\ \mathcal{M}_m}$. By writing $\displaystyle a_0=ms+r$ with $0 \leq r <m$, it follows that $[\alpha]=[r]$,
that is, $\alpha \in [r]$, which proves the result. \end{proof}

\begin{lemma} \label{basis} The rank of $\mathcal{M}_m$ is $p$. \end{lemma} \begin{proof} If $p\mid m$, then $\{m/p,m-\theta(t),\ldots,m-\theta^{p-1}(t)\}$ is a $\mathbb{Z}$-basis of $\mathcal{M}_m$. If $p\nmid m$, then $\{m,m-\theta(t),\ldots,m-\theta^{p-1}(t)\}$ is a $\mathbb{Z}$-basis of $\mathcal{M}_m$. Therefore, in both cases, the rank is $p$. \end{proof}


\subsubsection{Case $p \mid m$.} Suppose that $p$ is a divisor of $m$. As pointed in Lemma \ref{basis}, the submodule $\mathcal{M}_m$ of $\mathcal{O}_{\mathbb{K}}$ has basis $\{m/p,m-\theta(t),\ldots,m-\theta^{p-1}(t)\}$. In the next proposition, we calculate the trace form associated to $\mathcal{M}_m$ in relation with this basis:

\begin{proposition} \label{trace-1} If $\alpha = a_0\frac{m}{p}+\sum_{i=1}^{p-1}a_i(m-\theta^i(t))\in\mathcal{M}_m$, for some integer $a_0, a_1, \dotsc, a_{p-1}$, then
\begin{equation} \label{eq-2} Tr_{\mathbb{K}}(\alpha^2) = p\left(\left(\frac{a_0m}{p}+m\sum_{i=1}^{p-1}a_i\right)^2 +u\left( p\sum_{i=1}^{p-1}a_i^2-\left(\sum_{i=1}^{p-1} a_i\right)^2\right)\right) \end{equation} where $u=n/p^2$. \end{proposition}

\begin{proof} Developing the expression of $\alpha^2$, we have the following:
$$ \begin{array}{lll} \alpha^2 &=& \frac{a_0^2m^2}{p^2}+2\frac{a_0m}{p}\sum_{i=1}^{p-1}a_i(m-\theta^i(t))+\left(\sum_{i=1}^{p-1}a_i(m-\theta^i(t))\right)^2 \vspace{.3cm}\\ &=& \displaystyle \frac{a_0^2m^2}{p^2}+2\frac{a_0m}{p}\sum_{i=1}^{p-1}a_i(m-\theta^i(t)) + \sum_{i=1}^{p-1}a_i^2(m-\theta^i(t))^2\\
	&&\displaystyle +2\sum_{1\leq i<j\leq p-1} a_ia_j(m-\theta^i(t))(m-\theta^j(t)) \vspace{.3cm}\\ &=&  
	\displaystyle \frac{a_0^2m^2}{p^2}+\frac{2a_0m}{p}\sum_{i=1}^{p-1}a_i(m-\theta^i(t)) + \sum_{i=1}^{p-1}a_i^2(m^2-2m\theta^i(t)+\theta^i(t)\theta^i(t))\vspace{.3cm}\\
	&& \displaystyle+2\sum_{1\leq i<j\leq p-1} a_ia_j(m^2-m\theta^i(t)-m\theta^j(t)+\theta^i(t)\theta^j(t)).\end{array}$$ 
Thus, the trace form is given by
$$\begin{array}{lll} Tr_{\mathbb{K}}(\alpha^2) &=& \displaystyle \frac{a_0^2m^2}{p^2} Tr_{\mathbb{K}}(1)+2\frac{a_0m}{p}\sum_{i=1}^{p-1}a_iTr_{\mathbb{K}}(m-\theta^i(t))\\ & &  \displaystyle + \sum_{i=1}^{p-1}a_i^2Tr_{\mathbb{K}}(m^2-2m\theta^i(t)+\theta^i(t)\theta^i(t)) \vspace{.3cm}\\ && \displaystyle +2\sum_{1\leq i<j\leq p-1} a_ia_jTr_{\mathbb{K}}(m^2-m\theta^i(t)-m\theta^j(t)+\theta^i(t)\theta^j(t)). \end{array}$$ 
Since $Tr_{\mathbb{K}}(1)=p$ and $Tr_{\mathbb{K}}(\theta^i(t))=0$, for $i=1,2,\ldots,p-1$, it follows from Equation \eqref{eq-0} that
$$\begin{array}{lll} Tr_{\mathbb{K}}(\alpha^2) &=& \displaystyle p\left(\frac{a_0^2m^2}{p^2}\right)+2pm\frac{a_0m}{p}\sum_{i=1}^{p-1}a_i + \sum_{i=1}^{p-1}a_i^2\left(pm^2+\frac{n(p-1)}{p}\right) \\ & & \displaystyle + 2\sum_{1\leq i<j\leq p-1} a_ia_j\left(pm^2-\frac{n}{p}\right)  \end{array}$$
and so
\begin{equation*}
	Tr_{\mathbb{K}}(\alpha^2) = p\left(\frac{a_0^2m^2}{p^2}+\frac{2a_0m^2}{p}\sum_{i=1}^{p-1}a_i + \sum_{i=1}^{p-1}a_i^2(m^2+u(p-1))+2\sum_{1\leq i<j\leq p-1} a_ia_j(m^2-u) \right).
\end{equation*}
From this and since $\sum_{i=1}^{p-1}a_i^2 + 2\sum_{1\leq i<j\leq p-1}a_ia_j = \left(\sum_{i=1}^{p-1}a_i\right)^2$, it follows that
\begin{equation*}
	Tr_{\mathbb{K}}(\alpha^2) =  p\left(\frac{(a_0m)^2}{p^2}+\frac{2a_0m^2}{p}\sum_{i=1}^{p-1}a_i +(m^2-u)\left(\sum_{i=1}^{p-1}a_i\right)^2+up\sum_{i=1}^{p-1} a_i^2\right).
\end{equation*}
Finally, since
$$\left(\frac{a_0}{p}+\sum_{i=1}^{p-1}a_i\right)^2 = \frac{a_0^2}{p^2}+2\frac{a_0}{p}\sum_{i=1}^{p-1}a_i+\left(\sum_{i=1}^{p-1}a_i\right)^2,$$
then
$$Tr_{\mathbb{K}}(\alpha^2) = p\left(\left(\frac{a_0m}{p}+m\sum_{i=1}^{p-1}a_i\right)^2 -u\left(\sum_{i=1}^{p-1}a_i\right)^2+up\sum_{i=1}^{p-1} a_i^2\right),$$
that is,
$$Tr_{\mathbb{K}}(\alpha^2) = p\left(\left(\frac{a_0m}{p}+m\sum_{i=1}^{p-1}a_i\right)^2 +u\left( p\sum_{i=1}^{p-1}a_i^2-\left(\sum_{i=1}^{p-1} a_i\right)^2\right)\right),$$
which proves the result. \end{proof}


In order to calculate the minimum norm of the lattice $\Lambda_m=\sigma(\mathcal{M}_m)$, we will now compute the minimum of $Tr_{\mathbb{K}}(\alpha^2)$, for $0\neq\alpha\in \mathcal{M}_m$, considering the Equation \eqref{eq-2}. For this purpose, consider the quadratic form $Q_1:\mathbb{Z}\times\mathbb{Z}^{p-1}\rightarrow\mathbb{Z}$ given by $$Q_1(a_0,(a_1,\ldots,a_{p-1})) = \left(\frac{a_0m}{p}+m\sum_{i=1}^{p-1}a_i\right)^2 $$ and the quadratic form $Q_2:\mathbb{Z}^{p-1}\rightarrow \mathbb{Z}$ given by $$ Q_2(a_1,\ldots,a_{p-1}) = p\sum_{i=1}^{p-1}a_i^2-\left(\sum_{i=1}^{p-1}a_i\right)^2.$$
So, Proposition \ref{trace-1} provides \begin{equation} \label{eq-4} Tr_{\mathbb{K}}(\alpha^2) = pQ_1(a_0,(a_1,\ldots,a_{p-1}))+uQ_2(a_1,\ldots,a_{p-1}), \end{equation} for each $\alpha=a_0(m/p)+a_1(m-\theta(t))+\cdots+a_{p-1}(m-\theta^{p-1}(t))\in \mathcal{M}_m$, with $a_i\in\mathbb{Z}$, for $i=0,1,\ldots,p-1$.

\begin{proposition} \label{prop-1} $\displaystyle \min_{0\neq \alpha\in \mathcal{M}_m} Tr_{\mathbb{K}}(\alpha^2) = \min\left\{\frac{m^2}{p},up(p-1)\right\}.$ \end{proposition} \begin{proof} From \cite[Corollary 11]{ibilce-1}, $Q_2(a_1,\ldots,a_{p-1})=0$ if and only if $a_1=\ldots=a_{p-1}=0$. Thus, from Equation \eqref{eq-4}, the minimum of $Tr_{\mathbb{K}}(\alpha^2)$ is $m^2/p$, since the minimum of $Q_1$ with this condition is equal to $m^2/p^2$, which is achieved only by setting $a_0=\pm 1$. If $Q_2(a_1,\ldots,a_{p-1})>0$, from \cite[Corollary 11]{ibilce-1}, the minimum of $Q_2$ is $p-1$, which is achieved by the vectors $\pm(1,\ldots,1)$ and by the permutations of $\pm(1,0,\ldots,0)$. In this case, the minimum of $Q_1$ is zero, which is achieved only for $a_0=0$, and, thus, from Equation \eqref{eq-4} it follows that the minimum of $Tr_{\mathbb{K}}(\alpha^2)$ is $up(p-1)$. Therefore, the minimum of $Tr_{\mathbb{K}}(\alpha^2)$ for $0\neq\alpha\in M_m$ is $\min\{m^2/p,up(p-1)\}$. \end{proof}

As shown in the proof of Proposition \ref{prop-1}, the value $m^2/p$ is achieved for $\alpha=\pm m$ and the value $up(p-1)$ is achieved for $\alpha=\pm(m-\theta^i(t))$, with $i=1,2,\ldots,p-1$, and $\alpha=\pm\sum_{i=1}^{p-1}(m-\theta^i(t))$. Furthermore,
Proposition \ref{prop-1} and Equation \eqref{densidade} provides that the center density of the algebraic lattice $\Lambda_m=\sigma(\mathcal{M}_m)$ is given by $$\delta(\Lambda_m) = \frac{(\min\{m^2/p,up(p-1)\})^{p/2}}{2^pn^{\frac{p-1}{2}}m},$$
where $\Lambda_m = \sigma(\mathcal{M}_m)$. 

\subsubsection{Case: $p \nmid m$}\label{subsec_main}  Suppose that $p$ is not a divisor of $m$. As shown in the proof of Lemma \ref{basis}, the submodule $\mathcal{M}_m$ of $\mathcal{O}_{\mathbb{K}}$ has basis $\{m,m-\theta(t),\ldots,m-\theta^{p-1}(t)\}$. Let 
\[
\alpha=a_0m+\sum_{i=1}^{p-1}a_i(m-\theta^i(t))\in \mathcal{M}_m,
\]
with $a_0,a_1,\ldots,a_{p-1}\in\mathbb{Z}$. From \cite[Proposition 8]{ibilce-1}, it follows that 
\begin{equation} Tr_{\mathbb{K}}(\alpha^2) = p\left[m^2\left(a_0+\sum_{i=1}^{p-1}a_i\right)^2 + u \left[p\sum_{i=1}^{p-1} a_i^2 - \left(\sum_{i=1}^{p-1}a_i\right)^2 \right]\right], \end{equation}  where $u=n/p^2$.
From \cite[Theorem 12]{ibilce-1}, it follows that \[ \min_{0\neq \alpha\in \mathcal{M}_{m,0}} Tr_\mathbb{K}(\alpha^2) = \min\{pm^2,~up(p-1)\}. \] Furthermore, the center density of the algebraic lattice $\Lambda_m=\sigma(\mathcal{M}_m)$ is given by $$\delta(\Lambda_m) = \frac{(\min\{pm^2,up(p-1)\})^{p/2}}{2^pn^{\frac{p-1}{2}}m},$$
where $\Lambda_m = \sigma(\mathcal{M}_m)$.

Observe that, in general, $\Lambda_m$ is not a well-rounded lattice. 
However, consider the its submodule \[ \mathcal{M}=\{a_0(m-t)+a_1(m-\theta(t))+\cdots+a_{p-1}(m-\theta^{p-1}(t)): \ a_0,a_1,\ldots,a_{p-1}\in\mathbb{Z}\}. \] The module $\mathcal{M}$ has rank $p$ and $\mathcal{M}\subsetneq \mathcal{M}_m$, since $m\in\mathcal{M}_m$ and $m\notin \mathcal{M}$.

In the following, we compute the trace of $\alpha^2$, for all $\alpha$ in this $\mathbb{Z}$-submodule:

\begin{proposition} \label{trace-2} If $\alpha=a_0(m-t)+a_1(m-\theta(t))+\cdots+a_{p-1}(m-\theta^{p-1}(t))\in \mathcal{M}$, with $a_0,a_1,\ldots,a_{p-1}\in\mathbb{Z}$, then  \begin{equation} \label{eq-1}  Tr_{\mathbb{K}}(\alpha^2) = p\left(up\sum_{i=0}^{p-1}a_i^2 + (m^2-u)\left(\sum_{i=0}^{p-1}a_i\right)^2\right). \end{equation} \end{proposition} \begin{proof} The expression of $\alpha^2$ is given by
$$\begin{array}{lll} \alpha^2 &=& \displaystyle \sum_{i=0}^{p-1}a_i^2(m-\theta^i(t))^2+2\sum_{0\leq i<j\leq p-1} a_ia_j(m-\theta^i(t))(m-\theta^j(t))\vspace{.2cm}\\ 
&=& \displaystyle \sum_{i=0}^{p-1}a_i^2(m^2-2m\theta^i(t)+\theta^i(t)\theta^i(t)) \\
&& \displaystyle + 2\sum_{0\leq i<j\leq p-1} a_ia_j(m^2-m\theta^i(t)-m\theta^j(t)+\theta^i(t)\theta^j(t)). \end{array}$$
Thus, \[ \begin{array}{lll} Tr_{\mathbb{K}}(\alpha^2) &=& \displaystyle \sum_{i=0}^{p-1}a_i^2Tr_{\mathbb{K}}(m^2-2m\theta^i(t)+\theta^i(t)\theta^i(t)) \\
&& \displaystyle + 2\sum_{0\leq i<j\leq p-1} a_ia_jTr_{\mathbb{K}}(m^2-m\theta^i(t)-m\theta^j(t)+\theta^i(t)\theta^j(t)). \end{array} \]
Since $Tr_{\mathbb{K}}(\theta^i(t))=0$, for $i=0,1,\ldots,p-1$, then 
\begin{equation*}
Tr_{\mathbb{K}}(\alpha^2) = \sum_{i=0}^{p-1}a_i^2(Tr_{\mathbb{K}}(m^2)+Tr_{\mathbb{K}}(\theta^i(t)\theta^i(t))) + 2\sum_{0\leq i<j\leq p-1} a_ia_j(Tr_{\mathbb{K}}(m^2)+Tr_{\mathbb{K}}(\theta^i(t)\theta^j(t))).  
\end{equation*}
Also, $Tr_{\mathbb{K}}(\theta^i(t)^2)=Tr_{\mathbb{K}}(t^2)$ and $Tr_{\mathbb{K}}(\theta^i(t)\theta^j(t))=Tr_{\mathbb{K}}(t\theta^{j-i}(t))$ for $j>i$. So, 
\begin{equation*}
Tr_{\mathbb{K}}(\alpha^2) = \displaystyle \sum_{i=0}^{p-1}a_i^2(Tr_{\mathbb{K}}(m^2)+Tr_{\mathbb{K}}(t^2)) + 2\sum_{0\leq i<j\leq p-1} a_ia_j(Tr_{\mathbb{K}}(m^2)+Tr_{\mathbb{K}}(t\theta^{j-i}(t)).
\end{equation*}
From Equation \eqref{eq-0}, it follows that 
$$\begin{array}{lll} Tr_{\mathbb{K}}(\alpha^2) &=& \displaystyle \sum_{i=0}^{p-1}a_i^2\left(m^2p+\frac{n(p-1)}{p}\right) + 2\sum_{0\leq i<j\leq p-1} a_ia_j\left(m^2p+\frac{-n}{p}\right) \vspace{.2cm}\\ &=& \displaystyle \left(m^2p+\frac{n(p-1)}{p}\right)\sum_{i=0}^{p-1}a_i^2+2\left(m^2p-\frac{n}{p}\right)\sum_{0\leq i<\leq p-1}a_ia_j\vspace{.2cm}\\ &=& \displaystyle p\left( \left(m^2+u(p-1)\right)\sum_{i=0}^{p-1}a_i^2+2(m^2-u)\sum_{0\leq i<j\leq p-1}a_ia_j\right). \end{array}$$
Since $\displaystyle \left(\sum_{i=0}^{p-1}a_i\right)^2 = \sum_{i=0}^{p-1}a_i^2 + 2\sum_{0\leq i<j\leq p-1}a_ia_j$, denoting $u=n/p^2$, we have that
$$Tr_{\mathbb{K}}(\alpha^2) = p\left(up\sum_{i=0}^{p-1}a_i^2 + (m^2-u)\left(\sum_{i=1}^{p-1}a_i\right)^2\right),$$
which proves the result. \end{proof}

Equation \eqref{eq-1} can be rewritten as $$Tr_{\mathbb{K}}(\alpha^2) = p\left(uQ_1(a_0,a_1,\ldots,a_{p-1})+m^2Q_2(a_0,a_1,\ldots,a_{p-1}) \right),$$
where $Q_1:\mathbb{Z}^p\rightarrow \mathbb{Z}$ and  $Q_2:\mathbb{Z}^p\rightarrow \mathbb{Z}$ are the quadratic forms defined by
$$Q_1(a_0,a_1,\ldots,a_{p-1}) = p\sum_{i=0}^{p-1}a_i^2-\left(\sum_{i=0}^{p-1}a_i\right)^2$$
and
$$Q_2(a_0,a_1,\ldots,a_{p-1}) = \left(\sum_{i=0}^{p-1}a_i\right)^2.$$
The minimum of the trace form $Tr_\mathbb{K}(\alpha^2)$ is given below:

\begin{proposition} \label{prop-2} $\displaystyle \min_{0\neq \alpha\in \mathcal{M}} Tr_{\mathbb{K}}(\alpha^2) = p(u(p-1)+m^2).$
\end{proposition} \begin{proof} Proposition \ref{trace-2} provides $$Tr_{\mathbb{K}}(\alpha^2) = p\left(upQ_1(a_0,a_1,\ldots,a_{p-1}) + m^2Q_2(a_0,a_1,\ldots,a_{p-1})\right),$$ for each 
\[
\alpha=a_0(m-t)+a_1(m-\theta(t))+\cdots+a_{p-1}(m-\theta^{p-1}(t))\in \mathcal{M},
\]
with $a_0,a_1,\ldots,a_{p-1}\in\mathbb{Z}$. From \cite[Corollary 11]{ibilce-1}, the minimum of the quadratic form $Q_1(a_0,a_1,\ldots,a_{p-1})$ is $p-1$, which is achieved by the permutations of vectors $\pm(0,1,1,\ldots,1)$ and by the permutations of $\pm(1,0,\ldots,0)$. In turn, the minimum of the quadratic form $Q_2(a_0,a_1,\ldots,a_{p-1})$ is $1$, which is achieved by the permutations of $\pm(1,0,\ldots,0)$. This proves the proposition.\end{proof}

By Proposition \ref{prop-2} and Equation \eqref{densidade}, we have that the center density of the algebraic lattice $\Lambda_m=\sigma(\mathcal{M})$ is equal to
$$\delta(\Lambda_m) = \frac{(p(u(p-1)+m^2))^{p/2}}{2^pn^{\frac{p-1}{2}}m}.$$

\begin{proposition}  $\sigma(\mathcal{M})$ is a well-rounded lattice with a minimal basis \[\{m-t,\,m-\theta(t),\,\dotsc,\,m-\theta^{p-1}(t)\}.\] \end{proposition} \begin{proof} For $i=0,1,\ldots,p-1$, it follows from Equation \eqref{eq-0} that
\begin{align*}
    Tr_{\mathbb{K}}(m-\theta^i(t))^2
    &= Tr_{\mathbb{K}}(m^2)-2mTr_{\mathbb{K}}(\theta^i(t))+Tr_{\mathbb{K}}(t^2) \\
    &= pm^2+(n(p-1))/p=p(u(p-1)+m^2),
\end{align*}
which proves the result. \end{proof}

\subsection*{Acknowledgments}
The authors thank the reviewers for their valuable suggestions and the funding received from CNPq under Grant No. 405842/2023-6, from CAPES-PRINT-UNESP and from FAPESP 2022/02303-0.



\EditInfo{September 10, 2024}{March 14, 2025}{Lenny Fukshansky}

\end{document}